\documentclass[11pt, reqno]{amsart}

\usepackage[utf8]{inputenc}
\usepackage[T1]{fontenc}
\usepackage[english]{babel}
\usepackage{amsmath, amsthm}
\usepackage{ae}
\usepackage{icomma}
\usepackage{units}
\usepackage{color}
\usepackage{graphicx}
\usepackage{bbm}
\usepackage{caption}
\usepackage{array}
\usepackage[hmarginratio=1:1]{geometry}
\usepackage[hyphens]{url}
\usepackage[pdfpagelabels=false, hidelinks]{hyperref}
\usepackage{mathrsfs}
\usepackage{amssymb}
\usepackage{placeins} 
\usepackage{tikz}
\usetikzlibrary{patterns}
\usepackage{float} 
\usepackage[nodayofweek]{datetime}
\usepackage{enumitem}
\usepackage{mathtools}
\usepackage[numbers, sort]{natbib}
\usepackage{dsfont}
\usepackage{comment}
\usepackage{ifthen}

\newcommand{\R}{\ensuremath{\mathbb{R}}}

\DeclareMathOperator{\dist}{\textnormal{dist}}

\let\eps\varepsilon

\let\phi\varphi

\DeclareMathOperator{\sgn}{sgn}

\newcommand{\pps}{\hspace{0.75pt}}

\newtheorem{theorem}{Theorem}[section]

\newtheorem{lemma}[theorem]{Lemma}

\newtheorem{corollary}[theorem]{Corollary}

\numberwithin{theorem}{section}
\numberwithin{definition}{section}
\theoremstyle{remark}
\newtheorem{remark}[theorem]{Remark}

\begin{document}

\title[A sharp Hermite--Hadamard inequality]{A sharp multidimensional Hermite--Hadamard inequality}

\author[S. Larson]{Simon Larson}
\address{\textnormal{(S. Larson)} Department of Mathematics, Caltech, Pasadena, CA 91125, USA}
\email{larson@caltech.edu}

\thanks{\copyright\, 2020 by the author. This paper may be
reproduced, in its entirety, for non-commercial purposes.\\ Knut and Alice Wallenberg Foundation grant KAW~2018.0281 is acknowledged.}


\begin{abstract} 
    Let $\Omega \subset \R^d$, $d \geq 2$, be a bounded convex domain and $f\colon \Omega \to \R$ be a non-negative subharmonic function. In this paper we prove the inequality
    \begin{equation*}
       \frac{1}{|\Omega|}\int_\Omega f(x)\,dx \leq \frac{d}{|\partial\Omega|}\int_{\partial\Omega} f(x)\,d\sigma(x)\,.
     \end{equation*} 
     Equivalently, the result can be stated as a bound for the gradient of the Saint Venant torsion function. Specifically, if $\Omega \subset \R^d$ is a bounded convex domain and $u$ is the solution of $-\Delta u =1$ with homogeneous Dirichlet boundary conditions, then
     \begin{equation*}
       \|\nabla u\|_{L^\infty(\Omega)} < d\frac{|\Omega|}{|\partial\Omega|}\,.
     \end{equation*}
     Moreover, both inequalities are sharp in the sense that if the constant $d$ is replaced by something smaller there exist convex domains for which the inequalities fail. This improves upon the recent result that the optimal constant is bounded from above by~$d^{3/2}$ due to Beck et al.~\cite{beck_improved_2019}.
\end{abstract}

\maketitle

\section{Introduction and main results}

\subsection{Introduction} In this note we make several observations concerning multidimensional Hermite--Hadamard inequalities. These inequalities have been the subject of study in several recent articles, we refer to~\cite{beck_improved_2019,pasteczka_jensen-type_2018,steinerberger_hermite-hadamard_2018,lu_dimension-free_2019,hoskins_towards_2019}. In particular, we are interested in $c_d(\Omega)$ the optimal constant in the inequality
\begin{equation}\label{eq:HermiteHadamard}
  \frac{1}{|\Omega|}\int_{\Omega}f(x)\,dx \leq \frac{c_d(\Omega)}{|\partial\Omega|}\int_{\partial\Omega}f(x)\,d\sigma(x)\,,
\end{equation}
where $\Omega\subset \R^d$ is convex and bounded, and $f\colon \Omega \to \R$ is non-negative and subharmonic. If $\Omega$ is a ball $B$ the inequality with constant $c_d(B)=1$ is an easy consequence of the maximum and mean value principles, and~\eqref{eq:HermiteHadamard} can in a sense be regarded as a rigidity statement thereof. The case $d=1$ is the classical Hermite--Hadamard inequality~\cite{J1893,MR1505341}. Since the one-dimensional case is completely understood we assume throughout that $d\geq 2$.

Our main result concerns the smallest constant $c_d$ such that $c_d(\Omega)\leq c_d$ for all convex domains $\Omega \subset \R^d$,
\begin{equation}\label{eq:uniform constant}
  c_d=\sup\{ c_d(\Omega): \Omega\subset \R^d, \mbox{ convex and bounded}\}\,.
 \end{equation} 
 Specifically, we prove that $c_d=d$ and the supremum is not attained. In other words, for any convex $\Omega\subset \R^d$, $d\geq 2$, the strict inequality $c_d(\Omega)<d$ is valid. Moreover, in the two-dimensional case a result of M\'endez-Hern\'andez~\cite{mendez-hern_andez_brascamp-lieb-luttinger_2002} allows us to quantify this gap by showing that for any convex $\Omega \subset \R^2$
 \begin{equation}\label{eq: two-dimensional quantitative bound}
   c_2(\Omega) < 2\sqrt{1-c e^{-\frac{\pi}{2}\frac{D(\Omega)-r(\Omega)}{r(\Omega)}}}\,,
 \end{equation}
 where $D(\Omega)$ and $r(\Omega)$ denote the diameter and inradius of $\Omega$, respectively, and $c$ is a positive constant. In particular, we see that for $c_2(\Omega)$ to be close to $c_2=2$ the eccentricity of $\Omega$, i.e.\ $D(\Omega)/r(\Omega)$, needs to be sufficiently large. We suspect that similar estimates are valid also when $d\geq 3$.

\subsection{Main results}

In~\cite{steinerberger_hermite-hadamard_2018} Steinerberger proved that~\eqref{eq:HermiteHadamard} is valid for any convex $\Omega \subset \R^d$ and all non-negative convex functions $f$ with a constant $c_d(\Omega)\leq 2\pi^{-1/2}d^{d+1}$ (note that all convex functions are subharmonic). More recently it was proven by Beck et al.~\cite{beck_improved_2019} that the constant $c_d$ in~\eqref{eq:uniform constant} satisfies
\begin{equation}\label{eq: bounds by Beck et al.}
   d-1 \leq c_d \leq \begin{cases}
     d^{3/2} & \mbox{if } d \mbox{ is odd}\,,\\
     \frac{d(d+1)}{\sqrt{d+2}} & \mbox{if }d \mbox{ is even}\,.
   \end{cases}
 \end{equation} 
Our main result is that up to a slight change the proof of the upper bound in~\cite{beck_improved_2019} in fact yields that $c_d\leq d$. The new ingredient crucial for this improvement is an isoperimetric-type inequality for solutions of the heat equation on convex domains of fixed inradius proved by Ba\~nuelos and Kr\"oger~\cite{BanuelosKroger}. Furthermore, tracking equality cases throughout the proof enables us to construct a sequence of convex domains for which the inequality is asymptotically tight.
\begin{theorem}\label{thm:HermiteHadamard}
  Let $\Omega \subset \R^d$, $d\geq 2$, be a bounded convex domain. For all non-negative subharmonic functions $f\colon \Omega \to \R$ it holds that
  \begin{equation}\label{eq: HermiteHadamard ineq thm}
    \frac{1}{|\Omega|}\int_{\Omega}f(x)\,dx \leq \frac{d}{|\partial\Omega|}\int_{\partial\Omega}f(x)\,d\sigma(x)\,.
  \end{equation}
  Equality holds in~\eqref{eq: HermiteHadamard ineq thm} if and only if $f(x)\equiv 0$. Moreover, if the constant $d$ in the right-hand side were replaced by something smaller there exists a bounded convex domain $\Omega \subset \R^d$ and a non-negative subharmonic function $f\colon \Omega \to \R$ for which the inequality fails.
\end{theorem}

Our approach to proving Theorem~\ref{thm:HermiteHadamard}, which follows closely that of~\cite{beck_improved_2019}, is based on the following observation (see also~\cite{beck_improved_2019,lu_dimension-free_2019,hoskins_towards_2019,MR2083805,MR1801688}). Let $u_\Omega\colon \Omega \to \R$ denote the torsion function of $\Omega$, that is the solution to the boundary value problem
\begin{equation}\label{eq: def torsion function}
  \biggl\{\hspace{-5pt}\begin{array}{rl}
    -\Delta u_\Omega(x)=1\quad &\mbox{in }\Omega\,,\\[2pt]
    u_\Omega(x)=0 &\mbox{on }\partial\Omega\,.
  \end{array}
\end{equation}
Note that $u_\Omega(x)\geq 0$ in $\Omega$.
 For any $f\in H^1(\Omega)$
\begin{align*}
  \int_\Omega f(x)\,dx = \int_\Omega (-\Delta u_\Omega(x))f(x)\,dx = \int_\Omega u_\Omega(x)(-\Delta f(x))\,dx - \int_{\partial\Omega}\frac{\partial u_\Omega}{\partial \nu}(x)f(x)\,d\sigma(x)\,,
\end{align*}
where $\nu$ denotes the outward pointing unit normal. Since $u_\Omega \geq 0$ and vanishes on the boundary, we deduce for $f$ non-negative and subharmonic, $\Delta f(x)\geq 0$, that
\begin{equation}\label{eq: normal derivative HH}
  \int_\Omega f(x)\,dx \leq \Bigl\|\frac{\partial u_\Omega}{\partial \nu}\Bigr\|_{L^\infty(\partial\Omega)}\int_{\partial\Omega}f(x)\,d\sigma(x)\,.
\end{equation}
Note that we can get arbitrarily close to equality in the above by taking $f$ harmonic and $f|_{\partial\Omega}$ vanishing away from a neighbourhood of where the modulus of the normal derivative achieves its maximum.

As a consequence Theorem~\ref{thm:HermiteHadamard} follows as a corollary of the following result:
\begin{theorem}\label{thm: torsion gradient bound}
  Let $\Omega\subset \R^d$, $d\geq 2$, be a bounded convex domain and $u_\Omega$ solve~\eqref{eq: def torsion function}, then
  \begin{equation}\label{eq: gradient bound thm}
    \|\nabla u_\Omega \|_{L^\infty(\Omega)} < d\frac{|\Omega|}{|\partial\Omega|}\,.
  \end{equation}
  Moreover, if the constant $d$ in the right-hand side were replaced by something smaller there exists a bounded convex domain $\Omega \subset \R^d$ for which the inequality fails.
\end{theorem}

\begin{remark} A couple of remarks:
  \begin{enumerate}
    \item Firstly, since $u_\Omega$ vanishes on the boundary and
  \begin{equation*}
    \Delta |\nabla u_\Omega|^2 = 2\sum_{i,j=1}^d \Bigl(\frac{\partial^2 u_\Omega}{\partial x_i \partial x_j}\Bigr)^2 \geq 0
  \end{equation*}
  the maximum principle implies that $\bigl\|\frac{\partial u_\Omega}{\partial \nu}\bigr\|_{L^\infty(\partial\Omega)} = \|\nabla u_\Omega\|_{L^\infty(\Omega)}$. Thus any bound for $c_d(\Omega)$ implies a corresponding bound for $\|\nabla u_\Omega\|_{L^\infty(\Omega)}$.

  \item Secondly, an application of Greens identity yields a matching lower bound:
  \begin{equation}\label{eq: matching lower bound gradient}
    \|\nabla u_\Omega\|_{L^\infty(\Omega)}= \Bigl\|\frac{\partial u_\Omega}{\partial \nu}\Bigr\|_{L^\infty(\partial\Omega)} \geq \frac{1}{|\partial\Omega|}\int_{\partial\Omega}\Bigl|\frac{\partial u_\Omega}{\partial \nu}\Bigr|\,d\sigma(x) = \frac{|\Omega|}{|\partial\Omega|}\,,
  \end{equation}
  where equality holds if and only if $\Omega$ is a ball, by the classical overdetermined Serrin problem (see~\cite{NitschTrombetti,Serrin}).
  \end{enumerate}
\end{remark}

\subsection{A family of Hermite--Hadamard-type inequalities} In Section~\ref{sec:alpha inequalities} we consider a one-parameter family of inequalities containing the Hermite--Hadamard inequality~\eqref{eq:HermiteHadamard} as a special case. Specifically, we consider for $\alpha \leq d$ the inequality
\begin{equation}\label{eq: alpha ineq Omega}
  \int_\Omega f(x)\,dx \leq c_{d,\alpha}(\Omega)|\Omega|^{\frac{\alpha}d}|\partial\Omega|^{\frac{1-\alpha}{d-1}}\int_{\partial\Omega}f(x)\,d\sigma(x)\,,
\end{equation}
where as before $\Omega\subset \R^d$ is a convex domain and $f\colon \Omega \to \R$ is a non-negative subharmonic function. When $\alpha=d$ the inequality~\eqref{eq: alpha ineq Omega} is nothing but~\eqref{eq:HermiteHadamard}, while for $\alpha=1$ it reduces to an inequality studied in~\cite{lu_dimension-free_2019,steinerberger_hermite-hadamard_2018,beck_improved_2019,hoskins_towards_2019}. Again our main interest is towards upper bounds for $c_{d, \alpha}(\Omega)$. We note that for $\alpha>d$ no uniform bound can hold since by taking $f\equiv 1$ in~\eqref{eq: alpha ineq Omega} such a bound would imply a reverse isoperimetric inequality, which is a contradiction. However, for $\alpha <d$ a uniform bound can easily be deduced from the end-point case $\alpha=d$ and the isoperimetric inequality. In fact, as we shall see in Section~\ref{sec:alpha inequalities} any upper bound for $c_{d,\alpha_0}(\Omega)$ combined with the isoperimetric inequality implies an upper bound for $c_{d, \alpha}(\Omega)$ for all $\alpha<\alpha_0$.
In particular, the bounds for $\alpha=1$ in~\cite{lu_dimension-free_2019,steinerberger_hermite-hadamard_2018,beck_improved_2019,hoskins_towards_2019} imply bounds for all $\alpha <1$.

For this family of inequalities we are not able to say much concerning the optimal uniform constants
\begin{equation}\label{eq: supremum alpha}
  c_{d, \alpha}= \sup\{c_{d, \alpha}(\Omega): \Omega \subset \R^d, \mbox{ convex and bounded}\}\,.
\end{equation}
However, what we find interesting is that the dependence of $c_{d, \alpha}(\Omega)$ on the geometry appears very different for $\alpha<d$ compared to the end-point case $\alpha=d$. Indeed, we shall prove that when $\alpha<d$ the constant $c_{d,\alpha}(\Omega)$ becomes small if $\Omega$ has high eccentricity. We emphasize that this is fundamentally different form the behaviour we expect in the case $\alpha=d\geq 3$, and by~\eqref{eq: two-dimensional quantitative bound} know to be true when $\alpha=d=2$.

As in the case $\alpha=d$ our results can equivalently be phrased in terms of bounds for $\|\nabla u_\Omega\|_{L^\infty(\Omega)}$. Indeed, arguing as for $\alpha=d$ one concludes that
\begin{equation}\label{eq: alpha constant in terms of gradient}
  c_{d,\alpha}(\Omega)= |\Omega|^{-\frac{\alpha}{d}}|\partial\Omega|^{\frac{\alpha-1}{d-1}}\|\nabla u_\Omega\|_{L^\infty(\Omega)}\,.
\end{equation}
Our main result in this direction is the following:
\begin{theorem}\label{thm: alpha gradient bound}
  Let $\Omega\subset \R^d$, $d\geq 2$, be a bounded convex domain and $u_\Omega$ solve~\eqref{eq: def torsion function}. Then, for any $\alpha \leq d$,
  \begin{equation}\label{eq: alpha gradient bound thm}
    \|\nabla u_\Omega \|_{L^\infty(\Omega)} \leq \tilde c_{d, \alpha}\Bigl(\frac{r(\Omega)}{D(\Omega)}\Bigr)^{\!\frac{d-\alpha}{d(d-1)}}|\Omega|^{\frac{\alpha}{d}}|\partial\Omega|^{\frac{1-\alpha}{d-1}}\,,
  \end{equation}
  with $\tilde c_{d, \alpha}>0$ depending only on $d, \alpha$. Moreover, the power of $r(\Omega)/D(\Omega)$ is optimal.
\end{theorem}

While Theorem~\ref{thm: torsion gradient bound} tells us that for $\alpha=d$ the supremum~\eqref{eq: supremum alpha} is not attained and we expect any sequence of $\{\Omega_k\}_{k\geq 1}$, with $|\Omega_k|=1$, satisfying 
$$
\lim_{k\to \infty}|\partial\Omega_k|\pps\|\nabla u_{\Omega_k}\|_{L^\infty(\Omega_k)}= d
$$ 
to become unbounded in the limit, Theorem~\ref{thm: alpha gradient bound} tells us that the situation for $\alpha<d$ is different:
\begin{corollary}\label{cor: boundedness of maximizing sequence}
  For $\Omega\subset \R^d$, $d\geq 2$, denote by $u_\Omega$ the solution of~\eqref{eq: def torsion function}. For $\alpha < d$ any sequence of convex domains $\{\Omega_k\}_{k\geq 1}\subset \R^d$, with $|\Omega_k|=1$, satisfying
  \begin{align*}
    \liminf_{k\to \infty} |\partial\Omega_k|^{\frac{\alpha-1}{d-1}}\|\nabla u_{\Omega_k}\|_{L^\infty(\Omega_k)}>0\,,
  \end{align*}
  is uniformly bounded in the Hausdorff metric. In particular, up to translation any such sequence contains a subsequence converging with respect to the Hausdorff metric.
\end{corollary}
\begin{remark}
As a consequence of Corollary~\ref{cor: boundedness of maximizing sequence} we find that if $\{\Omega_k\}_{k\geq 1}$ is a maximizing sequence for $c_{d, \alpha}$ there exists a subsequence which, after translation, converges to a convex domain $\Omega^*$. Naturally, it is tempting to claim that the limit $\Omega^*$ realizes the supremum, $c_{d, \alpha}=c_{d,\alpha}(\Omega^*)$. However, although it is not very difficult to conclude that, up to passing to a subsequence,
$$
\nabla u_{\Omega_{k}} \to \nabla u_{\Omega^*} \quad \mbox{in } L^p(\R^d), \mbox{ for any } p<\infty\,,
$$  
we are at this point unable to deduce that 
$$
\lim_{k\to \infty}\|\nabla u_{\Omega_{k}}\|_{L^\infty(\Omega_{k})} = \|\nabla u_{\Omega^*}\|_{L^\infty(\Omega^*)}\,.
$$
Nevertheless, it is an interesting question to understand the shape of such limiting domains, and how their geometry depends on $\alpha$. The fact that makes this question particularly intriguing is that $\Omega^*$ is expected to be quite different from a ball. In fact, by arguing as in~\eqref{eq: matching lower bound gradient} and using the isoperimetric inequality one finds that balls are the unique minimizers of $c_{d, \alpha}(\Omega)$ for all $\alpha \leq d$,
\begin{align*}
  c_d(\Omega) = |\Omega|^{-\frac{\alpha}{d}}|\partial\Omega|^{\frac{\alpha-1}{d-1}}\|\nabla u_\Omega\|_{L^\infty(\Omega)} 
  \geq \biggl[\frac{|\Omega|^{\frac{1}{d}}}{|\partial\Omega|^{\frac{1}{d-1}}}\biggr]^{d-\alpha} 
  \geq 
  d^{\frac{\alpha-d}{d-1}}\omega_d^{\frac{\alpha-d}{d(d-1)}}\,,
\end{align*}
where equality holds if and only if $\Omega$ is a ball and $\omega_d$ denotes the volume of the $d$-dimensional unit ball.
For the case $d=2$ and $\alpha=1$ candidates for maximizing domains were obtained in~\cite{hoskins_towards_2019}.
\end{remark}

Before we move on we note that there are three cases of the shape optimization problem associated to~\eqref{eq: supremum alpha} which appear particularly natural:
\begin{enumerate}
  \item\label{itm: alpha=d} $\alpha =d$: Corresponding to the inequality $\|\nabla u_\Omega\|_{L^\infty(\Omega)}\leq c_{d}\frac{|\Omega|}{|\partial\Omega|}$. This case is of particular interest as the end-point and strongest inequality in the range, indeed for any $\alpha <\alpha'\leq d$ we have $c_{d,\alpha} \leq c_{d,\alpha'} \bigl(d\pps \omega_d^{1/d}\bigr)^{\frac{\alpha-\alpha'}{d-1}}$ (see~\eqref{eq: isoperimetric bound} below). 

  \item\label{itm: alpha=1} $\alpha=1$: Corresponding to the inequality $\|\nabla u_\Omega\|_{L^\infty(\Omega)}\leq c_{d,1}|\Omega|^{\frac{1}{d}}$, and by scaling the shape optimization problem of maximizing $\|\nabla u_\Omega\|_{L^\infty(\Omega)}$ with a measure constraint.

  \item\label{itm: alpha=0} $\alpha=0$: Corresponding to the inequality $\|\nabla u_\Omega\|_{L^\infty(\Omega)}\leq c_{d,0}|\partial\Omega|^{\frac{1}{d-1}}$, and by scaling the shape optimization problem of maximizing $\|\nabla u_\Omega\|_{L^\infty(\Omega)}$ with a perimeter constraint.
\end{enumerate}
Although the maximum of the gradient of the torsion function is a classical quantity in the Saint Venant theory of elasticity (the maximum shear stress), there is to the authors knowledge little known concerning these shape optimization problems. Apart from $\alpha=d$ and $\alpha=1$, which have recently been considered in the context of Hermite--Hadamard-type inequalities~\cite{hoskins_towards_2019,beck_improved_2019}, we are unaware of results in this direction. In particular, the problem when $\alpha=0$ appears not to have been studied. That being said there is a wide range of bounds for $\|\nabla u_\Omega\|_{L^\infty(\Omega)}$ under various assumptions on $\Omega$ available in the literature, we refer to~\cite{MR1244022} and references therein.

In the proof of our main result we shall see that the problem for $\alpha=d$ is closely related to maximizing $\|\nabla u_\Omega\|_{L^\infty(\Omega)}$ with a constraint on the inradius of $\Omega$. In fact, our proof of Theorem~\ref{thm: torsion gradient bound} relies on showing that this shape optimization problem is solved by the infinite slab, which follows by combining an inequality of Sperb with a result of Ba\~nuelos and Kr\"oger~\cite{sperb_maximum_1981,BanuelosKroger}. Although the results obtained in this paper essentially settle the shape optimization problem when $\alpha=d$ it would be interesting to obtain quantitative results similar in spirit to~\eqref{eq: two-dimensional quantitative bound} also when $d\geq 3$.

\section{Proof of Theorems~\ref{thm:HermiteHadamard} \&~\ref{thm: torsion gradient bound}}

By the argument in the previous section Theorem~\ref{thm:HermiteHadamard} follows as a consequence of Theorem~\ref{thm: torsion gradient bound}. Indeed, the inequality~\eqref{eq: HermiteHadamard ineq thm} follows from Theorem~\ref{thm: torsion gradient bound} and~\eqref{eq: normal derivative HH} as does the sharpness of the constant. Moreover, since~\eqref{eq: gradient bound thm} is strict, equality holds in~\eqref{eq: HermiteHadamard ineq thm} if and only if $f(x)\equiv 0$. Thus what remains is to prove Theorem~\ref{thm: torsion gradient bound}.

The proof of the inequality in Theorem~\ref{thm: torsion gradient bound} follows closely that of the upper bound provided in~\cite{beck_improved_2019}. However, instead of utilising the maximum principle to reduce the problem to considering an infinite slab of the same \emph{width} as $\Omega$, an application of a result of Ba\~nuelos and Kr\"oger reduces the problem to a slab of the same \emph{inradius} as $\Omega$. This allows us to remove the extra factor $\sim\sqrt{d}$ in the result of~\cite{beck_improved_2019} which arose as a consequence of using Steinhagen's inequality~\cite{Steinhagen} to bound the width $w(\Omega)$ in terms of the inradius $r(\Omega)$. The precise differences in the proofs will be explained in greater detail below.

\begin{proof}[Proof of Theorem~\ref{thm: torsion gradient bound}] We split the proof in two steps. In the first step the bound $c_{d}(\Omega)< d$ is established, while in the second step the sharpness of the bound is proved by explicit construction of a family of bounded convex sets $\{\Omega_\eta\}_{\eta}$ satisfying $\lim_{\eta \to \infty}c_d(\Omega_\eta)= d$.

\medskip

\noindent{\bf Step 1:} (Proof of the upper bound $c_d(\Omega)< d$)
Recall that for any bounded convex domain $\Omega\subset \R^d$ we have
\begin{equation}\label{eq: inradius bound}
  \frac{|\Omega|}{|\partial\Omega|}\leq r(\Omega)\leq d \frac{|\Omega|}{|\partial\Omega|}\,,
\end{equation}
see for instance~\cite[eq.~(13)]{LarsonJST}. Equality in the upper bound of~\eqref{eq: inradius bound} holds if and only if $\Omega$ is tangential to a ball (i.e.\ all its regular supporting hyperplanes are tangent to the same inscribed ball~\cite{Schneider1}). Also the lower bound is seen to be sharp by considering the sets $(-1, 1)\times (-R, R)^{d-1}$ as $R\to \infty$.

By~\eqref{eq: inradius bound} we have that
\begin{equation}\label{eq: gradient/inradius}
  \frac{|\partial\Omega|}{|\Omega|}\|\nabla u_\Omega\|_{L^\infty(\Omega)} \leq d\pps r(\Omega)^{-1}\|\nabla u_\Omega\|_{L^\infty(\Omega)}\,.
\end{equation}
We wish to maximize the right-hand side with respect to $\Omega$. We aim to show that
\begin{equation}\label{eq: gradient optimization fixed inradius}
  \sup\{r(\Omega)^{-1}\|\nabla u_\Omega\|_{L^\infty(\Omega)}: \Omega \subset \R^d, \mbox{ convex and bounded}\} = 1\,,
\end{equation}
and that the supremum is not achieved. Note that by scaling this is equivalent to maximizing $\|\nabla u_\Omega\|_{L^\infty(\Omega)}$ among all convex domains of a given inradius.

A classical inequality of Sperb allows us to bound the maximum of the gradient of $u_\Omega$ in terms of the function itself:
\begin{equation}\label{eq: Sperbs ineq}
  \|\nabla u_\Omega\|_{L^\infty(\Omega)}^2 \leq 2 \|u_\Omega\|_{L^\infty}\,.
\end{equation}
When $d=2$ this is~\cite[eq. (6.12)]{sperb_maximum_1981} while for $d\geq 3$ the inequality can be deduced in the same way from~\cite[Corollary~5.1]{sperb_maximum_1981}, note that Sperb considers $-\Delta u=2$ resulting in a different constant (see also~\cite{PhilippinSafoui} for non-smooth $\Omega$). It is easily checked by explicit calculation that equality in~\eqref{eq: Sperbs ineq} holds if $\Omega$ is the infinite slab $(-1, 1)\times \R^{d-1}$.

Let $p_\Omega(t, x, y)$ denote the heat kernel of the Dirichlet Laplacian on $\Omega$, then
\begin{equation*}
   u_\Omega(x) = \int_0^\infty \int_\Omega p_\Omega(t, x, y)\,dydt\,.
\end{equation*}
By integrating the bound of~\cite[Theorem~1]{BanuelosKroger} with respect to $t$, for any $x\in \Omega$,
\begin{equation}
  u_\Omega(x) = \int_0^\infty\int_\Omega p_\Omega(t, x, y)\,dydt \leq \int_0^\infty \int_{S_{r(\Omega)}}p_{S_{r(\Omega)}}(t, 0, y)\,dydt = u_{S_{r(\Omega)}}(0)\,,
\end{equation}
where $S_{r(\Omega)}=(-r(\Omega), r(\Omega))\times \R^{d-1}$ and equality holds if and only if $\Omega$ is an infinite slab and $\dist(x, \partial \Omega)=r(\Omega)$. Since $u_{S_{r(\Omega)}}(x) = \frac{r(\Omega)^2-x_1^2}{2}$, it follows that
\begin{equation}\label{eq: supremum bound fixed inradius}
  \|u_\Omega\|_{L^\infty(\Omega)}< \frac{r(\Omega)^2}{2}
\end{equation}
for any bounded convex $\Omega \subset \R^d$. For $d=2$ and with non-strict inequality this bound was proved by Sperb~\cite[eq.~(6.13)]{sperb_maximum_1981}.

Combining~\eqref{eq: Sperbs ineq} and~\eqref{eq: supremum bound fixed inradius} we find
\begin{equation*}
  r(\Omega)^{-1}\|\nabla u_\Omega\|_{L^\infty(\Omega)} \leq \sqrt{2}r(\Omega)^{-1}\|u_\Omega\|^{1/2}_{L^\infty(\Omega)}<1\,,
\end{equation*}
which proves that the supremum in~\eqref{eq: gradient optimization fixed inradius} is at most $1$ and that this value is not attained. In view of~\eqref{eq: gradient/inradius}, this completes the proof of the inequality in Theorem~\ref{thm: torsion gradient bound}. 

\medskip

Before moving on to the sharpness of the result we for the sake of comparison explain how the proof above differs from that given in~\cite{beck_improved_2019}. Specifically, the difference appears at the point in the proof where we appeal to the result of Ba\~nuelos and Kr\"oger. In~\cite{beck_improved_2019} the maximum principle and that, in appropriately chosen coordinates,
\begin{equation*}
  \Omega \subset S_{w(\Omega)/2}=\bigl(-\tfrac{w(\Omega)}{2}, \tfrac{w(\Omega)}{2}\bigr)\times \R^{d-1}
\end{equation*}
was used to bound
\begin{equation}\label{eq: maximum principle bound}
\|u_\Omega\|_{L^\infty(\Omega)} \leq \|u_{S_{w(\Omega)/2}}\|_{L^\infty(S_{w(\Omega)/2})} = \frac{w(\Omega)^2}{8}\,.
\end{equation}
As $r(\Omega)\leq w(\Omega)/2$ it is clear that this bound is in general weaker than~\eqref{eq: supremum bound fixed inradius}. Combining the inequalities~\eqref{eq: gradient/inradius},~\eqref{eq: Sperbs ineq}, and~\eqref{eq: maximum principle bound} with an application of Steinhagen's inequality~\cite{Steinhagen},
\begin{equation}\label{eq: Steinhagen}
   w(\Omega) \leq \begin{cases}
     2\sqrt{d}\, r(\Omega) & \mbox{if }d\mbox{ is odd}\\[2pt]
     2 \frac{d+1}{\sqrt{d+2}}\,r(\Omega)\hspace{10pt} \,   &\mbox{if }d\mbox{ is even}\,,
   \end{cases}
\end{equation}
completes the proof of the upper bound in~\eqref{eq: bounds by Beck et al.}. While Sperb's inequality~\eqref{eq: Sperbs ineq} is sharp for the infinite slab, equality in Steinhagen's inequality~\eqref{eq: Steinhagen} holds only for the regular $(d+1)$-simplex. These competing facts are what leads to the superfluous factor $\sim \sqrt{d}$ in the bound from~\cite{beck_improved_2019}. 

\medskip

{\noindent\bf Step 2:} (Proof of sharpness)
We shall construct a family of bounded convex sets $\{\Omega_\eta\}_{\eta}$ such that $\lim_{\eta \to \infty}c_d(\Omega_\eta)=d$. To achieve this we need to choose $\Omega_\eta$ such that we are arbitrarily close to equality in each part of our proof for the upper bound. Namely, we need to be close to equality in both the upper bound of~\eqref{eq: inradius bound} and almost attain the supremum in~\eqref{eq: gradient optimization fixed inradius}. As of yet we have not proved that the supremum~\eqref{eq: gradient optimization fixed inradius} is not less than $1$, but this is not too difficult. Indeed, one can consider the sequence $\Omega_R = (-1, 1)\times (-R, R)^{d-1}$ and show that equality holds in the limit $R\to \infty$. However, as we wish to prove that equality can be attained while simultaneously being close to equality in the upper bound of~\eqref{eq: inradius bound} we need to choose our domains a bit more carefully.

Let $\Omega_\eta$ be the $(d+1)$-simplex obtained by taking a regular $d$-simplex of sidelength $\eta \gg 1$ in the hyperplane $x_1=0$ centred at the origin and adding a final vertex at $(1, 0, \ldots, 0)$. In the planar case the construction is illustrated in Figure~\ref{fig: almost maximizer}. 
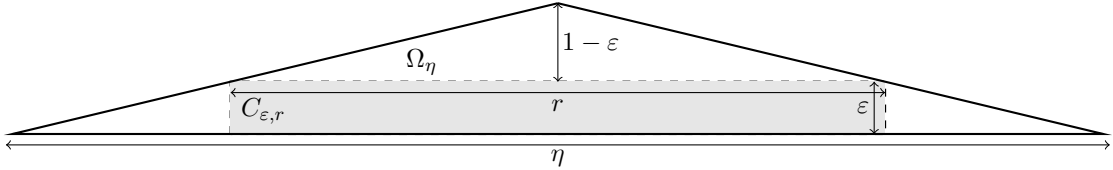
\begin{figure}[h]
  \centering
  \begin{tikzpicture}[scale=1.45]

  \edef \e {10};
  \edef \h {1.2};

  \draw[thick, dashed] (\e*0.2,0)--(\e*0.2,\h*0.4)--(\e*0.8,\h*0.4)--(\e*0.8,0)--(\e*0.2,0);
  \fill[gray!20] (\e*0.2,0)--(\e*0.2,\h*0.4)--(\e*0.8,\h*0.4)--(\e*0.8,0)--(\e*0.2,0);
  \draw[thick] (\e*0.5,\h)--(\e,0)--(0, 0)--(\e*0.5,\h);
  \draw[<->] (-0.06,-0.1)--(\e+0.06,-0.1);
  \draw[<->] (\e*0.5,\h*0.4)--(\e*0.5, \h);

  \draw[<->] (\e*0.2,\h*0.4-0.1)--(\e*0.8,\h*0.4-0.1);
  \draw[<->] (\e*0.8-0.1,0)--(\e*0.8-0.1,\h*0.4);

  \node at (\e*0.5,-0.22) {\scalebox{0.9}{$\eta$}};
  \node at (\e*0.5+0.3,\h*0.7) {\scalebox{0.9}{$1-\eps$}};
  \node at (\e*0.8-0.2,\h*0.2) {\scalebox{0.9}{$\eps$}};
  \node at (\e*0.5,\h*0.4-0.22) {\scalebox{0.9}{$r$}};
  \node at (\e*0.2+0.3,\h*0.18) {\scalebox{0.9}{$C_{\eps, r}$}};
  \node at (\e*0.375,\h*0.55) {\scalebox{0.9}{$\Omega_\eta$}};

  \end{tikzpicture}
  \caption{The almost maximizing domain $\Omega_\eta$ along with the inscribed box $C_{\eps, r}$. For any $\eps, r>0$ we can find $\eta$ sufficiently large so that $C_{\eps, r}\subset \Omega_\eta$.}
  \label{fig: almost maximizer}
\end{figure}

Since $\Omega_\eta$ is a $(d+1)$-simplex and thus tangential to a ball
\begin{equation*}
  \frac{|\partial \Omega_\eta|}{|\Omega_\eta|} = \frac{d}{r(\Omega_\eta)}\,.
\end{equation*}
In order to complete our proof we need to show that
\begin{equation*}
  r(\Omega_\eta)^{-1}\|\nabla u_\Omega\|_{L^{\infty}(\Omega_\eta)}=r(\Omega_\eta)^{-1}\Bigl\|\frac{\partial u_{\Omega_\eta}}{\partial \nu}\Bigr\|_{L^{\infty}(\partial\Omega_\eta)} \geq 1+ o(1)\,, \quad \mbox{as }\eta \to \infty\,.
\end{equation*}

For any $\eps>0$ and $r>0$ there exists an $\eta$ large enough so that
\begin{equation*}
  C_{\eps, r}=\bigl\{x\in \R^d: 0<x_1<1-\eps,\ |x_j| < r, j=2, \ldots, d\bigr\} \subset \Omega_\eta\,,
\end{equation*}
see Figure~\ref{fig: almost maximizer}.
By the maximum principle $0< u_{C_{\eps,r}}(x) \leq u_{\Omega_\eta}(x)$ for all $x\in C_{\eps, r}$. Since $0 \in \partial C_{\eps,r}\cap \partial \Omega_\eta$ and $u_{\Omega_\eta}(0)=u_{C_{\eps, r}}(0)=0$ it holds that
\begin{equation*}
  \Bigl|\frac{\partial u_{C_{\eps,r}}(0)}{\partial \nu}\Bigr| \leq \Bigl|\frac{\partial u_{\Omega_\eta}(0)}{\partial \nu}\Bigr|\,.
\end{equation*}
Consequently, for all $\eps>0, r>0$ there exists an $\eta$ large enough so that
\begin{equation*}
  r(\Omega_\eta)^{-1}\Bigl\|\frac{\partial u_{\Omega_\eta}}{\partial \nu}\Bigr\|_{L^{\infty}(\partial\Omega_\eta)}\geq r(\Omega_\eta)^{-1}\Bigl|\frac{\partial u_{C_{\eps,r}}(0)}{\partial \nu}\Bigr|\geq 2\Bigl|\frac{\partial u_{C_{\eps,r}}(0)}{\partial \nu}\Bigr|\,,
\end{equation*}
where we used $r(\Omega_\eta)\leq 1/2$. 

As $r\to \infty$ and $\eps \to 0$ the function $u_{C_{\eps, r}}$ converges to $u_S(x)= \frac{x_1(1-x_1)}{2}$ uniformly on any compact, where $S=\{x\in \R^d: 0<x_1<1\}$. Moreover, since $v=u_S-u_{C_{\eps, r}}$ is harmonic in $C_{\eps, r}$ and vanishes when $x_1=0$ we find that $\tilde v$ defined as the reflection of $v$ through $x_1=0$, that is for $x= (x_1, x')\in \R^d$ we set 
$$
\tilde v(x)= \sgn(x_1)v(|x_1|, x')\,,$$
is harmonic in 
$$
 \tilde C_{\eps, r}= \{x\in \R^d: |x_1|<1-\eps, |x_j|<r, j=2, \ldots, d\}\,
 $$
 Consequently, $\partial_j\tilde v(x)$ is harmonic for each $j=1, \ldots , d$ and thus the mean value principle and the divergence theorem yields, for any $0<\rho<1/2$,
\begin{align*}
  |\nabla \tilde v(0)| 
  =  \frac{1}{\omega_d\rho^d}\biggl|\int_{B_{\rho}(0)}\nabla \tilde v(x)\,dx\biggr| 
  =
  \frac{1}{\omega_d\rho^d}\biggl|\int_{\partial B_{\rho}(0)} \tilde v(x) \frac{x}{\rho}\,d\sigma(x)\biggr|
  \leq
  \frac{1}{\omega_d\rho^d}\int_{\partial B_{\rho}(0)} |\tilde v(x)|\,d\sigma(x)\,.
\end{align*}
Multiplying both sides by $\rho^{2d-1}$ and integrating from $0$ to $1/2$ with respect to $\rho$ we find that
\begin{equation*}
  |\nabla \tilde v(0)|\leq C_d \|\tilde v\|_{L^1(B_{1/2}(0))}\,.
\end{equation*}
Since $\tilde v$ converges to zero as $r\to \infty$ and $\eps \to 0$ we conclude that
\begin{equation*}
 \lim_{\substack{\eps \to 0\\ r\to \infty}} \Bigl|\frac{\partial u_{C_{\eps,r}}(0)}{\partial \nu}\Bigr| = \Bigl|\frac{\partial u_{S}(0)}{\partial \nu}\Bigr|=\frac{1}{2}\,.
\end{equation*}

By combining the above we find
\begin{equation*}
  c_d(\Omega_\eta)=\frac{|\partial\Omega_\eta|}{|\Omega_\eta|}\|\nabla u_{\Omega_\eta}\|_{L^\infty(\Omega_\eta)} = d\pps r(\Omega_\eta)^{-1}\|\nabla u_{\Omega_\eta}\|_{L^\infty(\Omega_\eta)} \geq d+o(1)\,, \quad \mbox{as }\eta \to \infty\,,
\end{equation*}
which proves that $c_d\geq d$ and therefore completes the proof of Theorem~\ref{thm: torsion gradient bound}.
\end{proof}

\section{A quantitative improvement when \texorpdfstring{$d=2$}{d=2}}

In the two-dimensional case a result of M\'endez-Hern\'andez~\cite{mendez-hern_andez_brascamp-lieb-luttinger_2002} actually allows one to strengthen both Theorem~\ref{thm:HermiteHadamard} and Theorem~\ref{thm: torsion gradient bound}. Namely, integrating the inequality of~\cite[Theorem~5.2]{mendez-hern_andez_brascamp-lieb-luttinger_2002} with respect to $t$ implies that the supremum of the torsion function of a bounded convex domain $\Omega\subset \R^2$ is bounded not only by that of the strip of same inradius but by the supremum of the torsion function in the truncated strip $((-r(\Omega), r(\Omega)) \times \R) \cap B_{D(\Omega)-r(\Omega)}(0)$.

Setting $R(\Omega)=(-r(\Omega), r(\Omega))\times (-D(\Omega)+r(\Omega), D(\Omega)-r(\Omega))$ and using the maximum principle one concludes that
\begin{equation}
  \|u_\Omega \|_{L^\infty} \leq \|u_{R(\Omega)}\|_{L^\infty} = u_{R(\Omega)}(0, 0)\,.
\end{equation}

The torsion function of a rectangle can be explicitly computed. Indeed, for $R_l = (-1/2, 1/2)\times (-l, l)$,
\begin{equation*}
  u_{R_l}(x_1, x_2) = \frac{1-4x_1^2}{8} - \frac{2}{\pi^3}\sum_{n\geq 1} \frac{1-(-1)^n}{n^3\cosh(n\pi l)}\cosh(n\pi x_2)\sin(n\pi (x_1+1/2))\,.
\end{equation*}

Since
\begin{equation*}
  u_{R(\Omega)}(x)= 4r(\Omega)^2u_{R_{\frac{D(\Omega)-r(\Omega)}{2r(\Omega)}}}(x/(2r(\Omega))) 
\end{equation*}
we find
\begin{align*}
  \|u_{R(\Omega)}\|_{L^\infty} 
  &= 4r(\Omega)^2 \|u_{R_{\frac{D(\Omega)-r(\Omega)}{2r(\Omega)}}}\|_{L^\infty} \\
  &= 
  4r(\Omega)^2u_{R_{\frac{D(\Omega)-r(\Omega)}{2r(\Omega)}}}(0, 0)\\
  &=
   \frac{r(\Omega)^2}{2} - \frac{8r(\Omega)^2}{\pi^3}\sum_{n\geq 1} \frac{1-(-1)^n}{n^3\cosh\bigl(n\pi \frac{D(\Omega)-r(\Omega)}{2r(\Omega)}\bigr)}\sin(n\pi/2)\\
   &=
   \frac{r(\Omega)^2}{2} - \frac{16r(\Omega)^2}{\pi^3}\sum_{k\geq 0} \frac{(-1)^k}{(2k+1)^3\cosh\bigl((2k+1)\pi \frac{D(\Omega)-r(\Omega)}{2r(\Omega)}\bigr)}\\
   &\leq
   \frac{r(\Omega)^2}{2} - \frac{16r(\Omega)^2}{\pi^3}\biggl[\frac{1}{\cosh\bigl(\pi \frac{D(\Omega)-r(\Omega)}{2r(\Omega)}\bigr)}-\frac{1}{9\cosh\bigl(3\pi \frac{D(\Omega)-r(\Omega)}{2r(\Omega)}\bigr)}\biggr]\,,
\end{align*}
where we in the last step used the fact that $x\mapsto \frac{1}{x^3 \cosh(x)}, x>0,$ is decreasing. 

Since $\cosh(x)$, $x>0$, is strictly increasing and satisfies $\frac{1}{\cosh(x)} \leq 2e^{-x}$,
\begin{equation*}
  \frac{1}{\cosh\bigl(\pi \frac{D(\Omega)-r(\Omega)}{2r(\Omega)}\bigr)}-\frac{1}{9\cosh\bigl(3\pi \frac{D(\Omega)-r(\Omega)}{2r(\Omega)}\bigr)} > \frac{8}{9\cosh\bigl(\pi \frac{D(\Omega)-r(\Omega)}{2r(\Omega)}\bigr)}\geq \frac{16}{9}e^{- \frac{\pi}{2} \frac{D(\Omega)-r(\Omega)}{r(\Omega)}}\,.
\end{equation*}
We conclude that
\begin{equation*}
  \|u_{R(\Omega)}\|_{L^\infty}  < \frac{r(\Omega)^2}{2}\Bigl[1- c e^{- \frac{\pi}{2} \frac{D(\Omega)-r(\Omega)}{r(\Omega)}}\Bigr]\,,
\end{equation*}
for any $c\leq \frac{2^9}{9\pi^3}$.

Consequently, the inequalities in Theorems~\ref{thm:HermiteHadamard} and~\ref{thm: torsion gradient bound} can for $d=2$ be strengthened to: For $\Omega\subset \R^2$ bounded and convex
\begin{equation*}
  \|\nabla u_\Omega\|_{L^\infty(\Omega)} < c_2(\Omega)\frac{|\Omega|}{|\partial\Omega|} \qquad \mbox{and}\qquad
  \frac{1}{|\Omega|}\int_\Omega f(x)\,dx \leq \frac{c_2(\Omega)}{|\partial\Omega|}\int_{\partial\Omega}f(x)\,d\sigma(x)\,,
\end{equation*}
for all non-negative subharmonic functions $f\colon \Omega \to \R$, with
\begin{equation*}
  c_2(\Omega)< 2\sqrt{1-ce^{-\pi/2 \frac{D(\Omega)-r(\Omega)}{r(\Omega)}}}
\end{equation*}
for some constant $c>0$.

\section{Upper bounds for \texorpdfstring{$c_{d,\alpha}(\Omega)$}{cda(Omega)}}
\label{sec:alpha inequalities}

We now turn our attention to the case $\alpha<d$. As mentioned in the introduction our result for $\alpha=d$ together with the isoperimetric inequality implies that $c_{d, \alpha}<\infty$. Indeed, by~\eqref{eq: alpha constant in terms of gradient}, Theorem~\ref{thm: torsion gradient bound}, and the isoperimetric inequality
\begin{equation}\label{eq: trivial bound cdalpha}
  c_{d, \alpha}(\Omega) 
  = \biggl[\frac{|\Omega|^{\frac{1}{d}}}{|\partial\Omega|^{\frac{1}{d-1}}}\biggr]^{d-\alpha} \frac{|\partial\Omega|}{|\Omega|}\|\nabla u_\Omega\|_{L^\infty(\Omega)} 
  \leq d^{\frac{\alpha-1}{d-1}}\omega_d^{\frac{\alpha-d}{d(d-1)}}\,.
\end{equation}
Note that this bound cannot possibly be sharp, indeed we have equality in the isoperimetric inequality if and only if $\Omega$ is a ball in which case we are far from equality in Theorem~\ref{thm: torsion gradient bound}. Utilizing a bound for $c_{d,1}$ proved in~\cite{beck_improved_2019} allows us to do better by a negative power of $d$:
\begin{lemma}\label{lem: alpha bound}
  Let $\Omega\subset \R^d$, $d\geq 2$, be a bounded convex domain and $u_\Omega$ solve~\eqref{eq: def torsion function}. Then, for $1\leq \alpha \leq d$,
  \begin{equation*}
    c_{d,\alpha} \leq d^{\frac{\alpha-1}{d-1}-\frac{d-\alpha}{2(d-1)}}\omega_d^{\frac{\alpha-d}{d(d-1)}}\,,
  \end{equation*}
  while for $\alpha \leq 1$
  \begin{equation*}
     c_{d,\alpha} \leq d^{\frac{\alpha-1}{d-1}-\frac{1}{2}}\omega_d^{\frac{\alpha-d}{d(d-1)}}\,.
  \end{equation*}
\end{lemma}
\begin{remark}
  For $d=2$ both bounds can be improved by utilizing the better bound for $c_{2,1}$ obtained in~\cite{hoskins_towards_2019}.
\end{remark}
\begin{proof}[Proof of Lemma~\ref{lem: alpha bound}]
  For any $\alpha< \alpha'\leq d$ the isoperimetric inequality and~\eqref{eq: alpha constant in terms of gradient} implies
  \begin{equation}\label{eq: isoperimetric bound}
     c_{d, \alpha}(\Omega) = \biggl[\frac{|\Omega|^{\frac{1}{d}}}{|\partial\Omega|^{\frac{1}{d-1}}}\biggr]^{\alpha'-\alpha} |\Omega|^{-\frac{\alpha'}{d}}|\partial\Omega|^{\frac{\alpha'-1}{d-1}}\|\nabla u_\Omega\|_{L^\infty(\Omega)} 
     \leq 
     \bigl(d \omega_d^{1/d}\bigr)^{\frac{\alpha-\alpha'}{d-1}} c_{d, \alpha'}(\Omega)\,.
   \end{equation} 
   Choosing $\alpha'=1$ and using the bound $c_{d,1}(\Omega)\leq \omega_d^{-1/d}d^{-1/2}$~\cite[Theorem~3]{beck_improved_2019} proves the second inequality of the lemma.

  To prove the first inequality we argue similarly. For any $\alpha_1 \leq \alpha \leq \alpha_2$ we have
  \begin{align*}
    c_{d, \alpha}(\Omega) 
    &= 
    \Bigl[|\Omega|^{-\frac{\alpha_1}{d}}|\partial\Omega|^{\frac{\alpha_1-1}{d-1}}\|\nabla u_\Omega\|_{L^\infty(\Omega)}\Bigr]^{\frac{\alpha_2-\alpha}{\alpha_2-\alpha_1}} 
    \Bigl[|\Omega|^{-\frac{\alpha_2}{d}}|\partial\Omega|^{\frac{\alpha_2-1}{d-1}}\|\nabla u_\Omega\|_{L^\infty(\Omega)}\Bigr]^{\frac{\alpha-\alpha_1}{\alpha_2-\alpha_1}}\\
    &= 
    c_{d, \alpha_1}(\Omega)^{\frac{\alpha_2-\alpha}{\alpha_2-\alpha_1}} 
    c_{d, \alpha_2}(\Omega)^{\frac{\alpha-\alpha_1}{\alpha_2-\alpha_1}}\,.
  \end{align*}
  Choosing $\alpha_1=1$, $\alpha_2=d$, and using the bounds of Theorem~\ref{thm: torsion gradient bound} and~\cite[Theorem~3]{beck_improved_2019} yields the desired bound, which completes the proof of Lemma~\ref{lem: alpha bound}.
\end{proof}

We turn to the proofs of Theorem~\ref{thm: alpha gradient bound} and Corollary~\ref{cor: boundedness of maximizing sequence}. 
\begin{proof}[Proof of Theorem~\ref{thm: alpha gradient bound}]
Our goal is to prove
\begin{equation}\label{eq: diameter decay proof}
  |\Omega|^{-\frac{\alpha}{d}}|\partial\Omega|^{\frac{\alpha-1}{d-1}}\|\nabla u_\Omega\|_{L^\infty(\Omega)} 
  =
  \biggl[\frac{|\Omega|^{\frac{1}{d}}}{|\partial\Omega|^{\frac{1}{d-1}}}\biggr]^{d-\alpha} \frac{|\partial\Omega|}{|\Omega|}\|\nabla u_\Omega\|_{L^\infty(\Omega)} \lesssim_{d,\alpha} \Bigl(\frac{r(\Omega)}{D(\Omega)}\Bigr)^{\frac{d-\alpha}{d(d-1)}}\,.
\end{equation}
By Theorem~\ref{thm: torsion gradient bound} and~\eqref{eq: matching lower bound gradient}
\begin{equation*}
 1\leq  \frac{|\partial\Omega|}{|\Omega|}\|\nabla u_\Omega\|_{L^\infty(\Omega)} < d\,,
\end{equation*}
therefore it suffices to prove
\begin{equation*}
  \frac{|\Omega|^{\frac{1}{d}}}{|\partial\Omega|^{\frac{1}{d-1}}} \lesssim_d \Bigl(\frac{r(\Omega)}{D(\Omega)}\Bigr)^{\frac{1}{d(d-1)}}\,.
\end{equation*}

By John's lemma~\cite{MR0030135} there exists an ellipsoid $E\subset \R^d$ such that $E$ is contained in $\Omega$ and the dilation of $E$ by a factor $d$ around its centre contains $\Omega$. Let $E$ be such an ellipsoid associated with $\Omega$ and denote by $r_1\leq r_2 \leq \ldots \leq r_d$ the lengths of the semi-axes of $E$. Note that $r_1$ and $r_d$ are comparable to $r(\Omega)$ and $D(\Omega)$, respectively. Moreover, since
\begin{equation*}
    |E|\sim_d \prod_{j=1}^d r_j \quad \mbox{and} \quad |\partial E|\sim_d \prod_{j=2}^d r_j
\end{equation*}  
the monotonicity of perimeter and volume under inclusion of convex sets implies
\begin{equation}\label{eq: isop. quotient eccentricity bound}
  \frac{|\Omega|^{\frac{1}{d}}}{|\partial\Omega|^{\frac{1}{d-1}}} \sim_d \frac{\prod_{j=1}^d r_j^\frac{1}{d}}{\prod_{j=2}^d r_j^{\frac{1}{d-1}}} = \prod_{j=2}^d \Bigl(\frac{r_1}{r_j}\Bigr)^{\frac{1}{d(d-1)}} \leq \Bigl(\frac{r_1}{r_d}\Bigr)^{\frac{1}{d(d-1)}} \sim_d \Bigl(\frac{r(\Omega)}{D(\Omega)}\Bigr)^{\frac{1}{d(d-1)}}
\end{equation}
as claimed. 

Since each step of the proof is sharp up to constants if $r_1 = \ldots = r_{d-1}$ the optimality of the exponent follows. This completes the proof of Theorem~\ref{thm: alpha gradient bound}.
\end{proof}

\begin{proof}[Proof of Corollary~\ref{cor: boundedness of maximizing sequence}]
  Let $\{\Omega_k\}_{k\geq 1}$ be as in the statement of the corollary. Since $|\Omega_k|=1$ we have that $r(\Omega_k)\leq \omega_d^{-1/d}$. Thus, by Theorem~\ref{thm: alpha gradient bound},
  \begin{equation*}
    |\partial\Omega_k|^{\frac{\alpha-1}{d-1}}\|\nabla u_{\Omega_k}\|_{L^\infty(\Omega_k)} \lesssim_{d,\alpha}  \Bigl(\frac{r(\Omega_k)}{D(\Omega_k)}\Bigr)^{\frac{d-\alpha}{d(d-1)}} \lesssim_{d, \alpha} D(\Omega)^{- \frac{d-\alpha}{d(d-1)}}
  \end{equation*}
  and by assumption
  \begin{align*}
    \liminf_{k\to \infty}|\partial\Omega_k|^{\frac{\alpha-1}{d-1}}&\|\nabla u_{\Omega_k}\|_{L^\infty(\Omega_k)} >0\,.
  \end{align*}
  Combining the above we find $D(\Omega_k)\lesssim_{d, \alpha} 1$. The existence of a convergent subsequence follows from the Blaschke selection theorem~\cite[Theorem~1.8.7]{Schneider1}.
\end{proof}

\medskip

\noindent{\bf Acknowledgements.} The author wishes to thank the anonymous referee for valuable comments and suggestions which significantly helped improve the quality of the manuscript.


\def\myarXiv#1#2{\href{http://arxiv.org/abs/#1}{\texttt{arXiv:#1\,[#2]}}}

\end{document}